\documentclass[12pt]{amsart}
\usepackage{amscd,amssymb}
\usepackage[graph,frame,poly,arc]{xy}  
\usepackage[plainpages,backref,urlcolor=blue]{hyperref}

\theoremstyle{plain}
\newtheorem{thm}[subsection]{Theorem}

\theoremstyle{definition}
\newtheorem{rk}[subsection]{Remark}

\newtheorem{ex}[subsection]{Example}

\newtheorem{question}[subsection]{Question}

\numberwithin{equation}{section}
\newcommand{\ko}{{\mathcal O}}
\newcommand{\OO}{{\mathcal O}}

\newcommand{\C}{\mathbb{C}}

\begin{document}

\title [1-forms on complete intersection curve singularities]
{On 1-forms on isolated complete intersection curve singularities\\
[1.0ex]{\em \tiny{To the memory of Egbert Brieskorn}}}
\author[Alexandru Dimca]{Alexandru Dimca$^1$}
\address{Universit\'e C\^ ote d'Azur, CNRS, LJAD, France }
\email{dimca@unice.fr}

\author[Gert-Martin Greuel]{Gert-Martin Greuel}
\address{Department of Mathematics,
University of Kaiserslautern, Germany}
\email{greuel@mathematik.uni-kl.de}

\thanks{$^1$ This work has been supported by the French government, through the $\rm UCA^{\rm JEDI}$ Investments in the Future project managed by the National Research Agency (ANR) with the reference number ANR-15-IDEX-01.}



\begin{abstract} We collect some classical results about holomorphic 1-forms of a reduced complex curve singularity. They are used to study the pull-back of holomorphic 1-forms on an isolated complete intersection curve singularity under the normalization morphism. We wonder whether the Milnor number $\mu$ and the Tjurina number $\tau$ of any isolated plane curve singularity satisfy the inequality $3\mu <4\tau$.
\end{abstract}
 
\maketitle



\section{Introduction} 

Consider a reduced complex curve singularity $(X,0) \subset (\C^N,0)$, defined by an ideal 
$I \subset \ko_{\C^N,0}$, with $r = r(X,0)$ branches. Let $\nu:(\bar{X},\bar{0}) \to X,0)$ be the normalization,
 where $(\bar{X},\bar{0})$ is the multi-germ consisting of $r$ smooth branches. We set
\[
\begin{array}{lll}
\ko & := & \ko_{X,0} = \ko_{\C^N,0}/I, \text{ the local ring of the germ  }(X,0);\\ 
\bar{\ko} & := & \nu_*\ko_{\bar{X},\bar{0}}, \text { the direct image of the local ring of the multi-germ } 
(\bar X,\bar 0) ;\\
\Omega & :=  & \Omega^1_{\C^N,0} / I \Omega^1_{\C^N,0} + \ko_{\C^N,0} dI, \text { the holomorphic 1-forms on } (X,0); \\
\bar{\Omega} & :=  & \nu_* \Omega^1_{\bar X,\bar 0}, \text { the direct image of the holomorphic 1-forms on } 
(\bar X,\bar 0); \\
\omega & := & \omega_{X,0} = Ext^{N-1}_{\ko_{\C^N,0}}(\ko,  \Omega^N_{\C^N,0}),  \text { the dualizing module of } (X,0); \\
T \Omega  & :=  &  H^0_{\{0\}}(\Omega),  \text { the torsion submodule of the } \ko \text{-module } \Omega.
\end{array}
\]

Let $d: \ko \to \Omega$ be the exterior derivation. We have the following  maps
$$d\ko \to \Omega \to \bar{\Omega} \to \omega,$$
where $d\ko \to \Omega$ is the inclusion, $\Omega \to \bar{\Omega}$ is given by the pull-back of forms under the morphism $\nu$, and $\bar{\Omega} \to \omega$ is the inclusion, if we identify the dualizing module $\omega$ with the module of Rosenlicht's regular differential forms as explained in \cite{BG80}.
Then the maps $d\ko \to \Omega$ and $\bar{\Omega} \to \omega$ are clearly injective and $T\Omega$ is the kernel of the map $\Omega \to \bar{\Omega} $ (cf. \cite{BG80}). We write  $\bar{\Omega} / \Omega$ for the cokernel of the map $ \Omega \to \bar{\Omega}$ and similarly for the other maps. These objects give rise to the following numerical invariants:  
 
 \[
\begin{array}{lll}
m & := & mt(X,0), \text{ the multiplicity of } (X,0); \\
\delta  & := &  \delta(X,0) = \text{dim}_\C (\bar{\ko} / \ko), \text{ the delta-invariant of } (X,0);\\
\mu & := & \mu(X,0) = \text{dim}_\C (\omega / d\ko),  \text{ the Milnor number of } (X,0); \\
\lambda  & := &  \lambda (X,0) =  \text{dim}_\C (\omega / \Omega);\\
\tau'  & := &  \tau ' (X,0) =  \text{dim}_\C (T\Omega); \\
 \tau & := &  \tau (X,0) =  \text{dim}_\C (T^1_{X,0}),  \text{ the Tjurina number of } (X,0). \\
\end{array}
\]

Here  $T^1_{X,0}$ is the tangent space of the base space of the semiuniversal deformation of $(X,0)$. 
If $(X,0)$ is a plane curve singularity with $ I = \langle f \rangle $, then $\mu = \dim_\C (\ko/J_f)$ (the classical Milnor number, cf. \cite{BG80, M68}) and $\tau 
= \dim_\C (\ko/\langle f \rangle+J_f)$, where $ J_f$ is the Jacobian ideal generated by the partials of $f$.

The aim of this note is to prove the following.

\begin {thm} \label{thmA}
Let  $(X,0)$ be a reduced complete intersection curve singularity. Then
the following hold.
\begin {enumerate}
\item $\tau = \tau'  = \lambda \geq  \delta +m -r,$ 
\item $ \tau - \delta = \dim_\C (\bar{\Omega} / \Omega)$. In particular, one has the equality
$$\dim_\C (\bar{\Omega} / \Omega)=\delta - r + 1$$
if and only if the singularity $(X,0)$ is weighted homogeneous.
\item $\tau >  \mu / 2$ if $(X,0)$ is not smooth.
\end {enumerate}

\end {thm}

In the second section we recall a number of classical results on isolated complete intersection singularities (due to the second author with several co-authors, and written partly in German), which are somewhat scattered in the literature and apparently not well known. We collect them here with reference to the original sources. 
In the third  section we give a quick proof of Theorem \ref{thmA} using the results quoted before and discuss its relations with similar results by Delphine Pol, see Remark  \ref{rk0}.
In the final section we discuss the possible values of the quotient
$\rho(X,0)=\mu(X,0)/ \tau(X,0)$ and ask whether $\rho(X,0)<4/3$ for any plane curve singularity.

\bigskip

We would like to thank Mathias Schulze for a useful remark, see Remark \ref{rk0}.

\section{The classical results} 
We start by recalling the following general result.
\begin {thm} \label{thm1} (\cite {BG80})\\ 
For a reduced curve singularity the following holds.
\begin {enumerate}
\item $\mu = 2\delta -r + 1$,
\item  $\mu \geq \lambda \geq \delta + m -r$,
\item  $\dim_\C (\Omega / d\ko) = \mu + \tau' - \lambda$,
\item $\dim_\C (\omega / \bar{\Omega}) = \delta$,
\item If $(X,0)$ is smoothable (e.g. if it is a complete intersection) then $ \tau' \geq  \lambda$, with equality iff  $ \mu = \dim_\C (\Omega / d\ko)$.
\end {enumerate}
\end {thm}

\begin {proof} All these claims are proved in  \cite {BG80}.
Indeed, (1) is Proposition 1.2.1, (2) is Lemma 6.1.2, (3) appears in the proof of Theorem 6.1.3, (4) in the proof of Proposition 1.2.1, while
(5) is Corollary 6.1.4 together with Corollary 6.1.6 of \cite {BG80}.
\end {proof}

When $(X,0)$ is a complete intersection, we have the following additional properties. Some of these results
are also reproduced in Looijenga's book \cite{L84}, see in particular  Section 8.C.
In the case of plane curves, the reader can also consult the introductory book \cite{W04}, in particular Section 11.6.

\begin {thm} \label{thm2} (\cite {G75} \cite {G80}), \cite {GMP85})\\ 
Let $(X,0)$ be a reduced complete intersection curve singularity. Then
\begin {enumerate}
\item $\mu = \dim_\C (\Omega / d\ko)$, $d\ko \cap T\Omega =0$,
\item $\tau = \tau' \leq \mu$,
\item $\tau =  \mu$ iff $(X,0)$ is quasihomogeneous.
\end {enumerate}
\end {thm}

\begin {proof} Indeed, (1) is Proposition 5.1, resp. Lemma 4.5 in \cite{G75} (for arbitrary positive dimensional isolated complete intersection singularities, resp. for complete intersections with arbitrary singularities, suitably modified), (2) is Satz 3.1(2a) in \cite{G80}. The claim (3) is Corollary 2.2 in \cite{GMP85} (where also a generalization to Gorenstein curves is proved), 
while the plane curve case goes back to K. Saito \cite{KS71}.
\end {proof}

\section{The proof of Theorem \ref{thmA}}

  The sequence
  \[
  0 \to T\Omega \to \Omega / d\ko \to \omega /d\ko  \to \omega / \Omega \to 0
  \]
is exact by Theorem \ref{thm2} (1) with dim$_{\C} (\Omega / d\ko) = \mu = $ dim${_\C} (\omega / d\ko)$. Hence $\tau' = $ dim$_{\C} (T\Omega)=$ dim$_{\C} ( \omega / \Omega)  = \lambda$. 
Claim (1) follows now from  Theorem \ref{thm2} (2) and  Theorem \ref{thm1} (2). 
The claim   (2) is a consequence of the exact sequence 
  \[
  0 \to \bar{\Omega} / \Omega \to \omega / \Omega \to \omega / \bar{\Omega} \to 0
  \]
  together with (1), Theorem \ref{thm1} (4) and the definition of $\lambda$. 
Using (1) and Theorem \ref{thm1} (1) we get  
$$\tau \geq  \delta +m -r = (\mu +r -1)/2 +m-r = \mu/2 +(m-1)/2 + (m-r)/2 >  \mu/2,$$ 
since $m\geq r$ and $m >1$ if $(X,0)$ is not smooth.

\begin{rk} \label{rk0}
It was drawn to our attention by Mathias Schulze that an alternative proof of the equality in Theorem \ref{thmA} (2) can be obtained from \cite[Proposition 3.31]{Pol14}. Assume that $(X,0)$ is irreducible for simplicity. 
Let 
 $f_1=\cdots =f_n=0$ be the equations for the germ $(X,0)$ in $(\C^{N},0)$, with $N=n+1$ and $f_i \in \OO_{\C^{N},0}$, for $i=1,...,n$.
Then  $\tau'$ is the codimension of the Jacobian ideal $J_X$ in $\OO$, where $J_X$ is the ideal of $\OO$, spanned by all the $n\times n $-minors in the Jacobian matrix $\left(\partial f_i/\partial x_j\right)_{i=1,n;j=0,n}$, see  \cite[Proposition 1.11(iii)]{G75}.
Delphine Pol shows  that one has the following equality  
$$J_X=g \cdot \frac{\nu^*(\Omega)}{dt},$$
in the local ring $\bar {\OO}=\C\{t\}$, where $g$ is a generator of the conductor ideal $C_X$.
Note that the codimension of $J_X$ in $\bar {\OO}$ is clearly by the above discussion $\tau+\delta$. Since $g$, regarded as an element of 
$\bar {\OO}=\C\{t\}$, has order $\mu$, it follows that the codimension of $g \cdot \frac{\nu^*(\Omega^1_{X,0})}{dt}$ in $\bar {\OO}$ is given by
$$ \mu+ \dim_\C( \bar{\Omega} / \Omega).$$
The claim follows from these formulas.  Note that both the proofs given, and the literature used, by Delphine Pol are quite different from ours.
\end{rk}

\begin {rk} \rm

(1) The equality $\tau = \tau'$ holds more generally if $(X,0)$ is Gorenstein, which follows from local duality. For an arbitrary reduced curve singularity $(X,0)$ the relation between $\tau$ and $\tau'$ is unclear. It is an old and still open question if for a non smooth $(X,0)$ we have always  $\tau > 0$  (i.e. $(X,0)$ is not rigid) and   $\tau' > 0$ (Berger's conjecture).

(2) For a plane curve singularity $(X,0)$ the expression $\tau - \delta$ appears also as the codimension of the extended tangent space to the orbit of the parametrization $(\bar{X},\bar 0) \to (\C^2,0)$ of $(X,0)$ by the action of the right-left group $\mathcal A$ of Mather (\cite [Proposition II.2.30(5)] {GLS07}).

\end {rk}

\section{A remark on the quotient $\mu/ \tau$}

Assume in this section that we are in the case of plane curve singularities, and we write $f_1=f$ to simplify our notation.
Let $M(f)=\OO_{\C^2,0}/J_f$ be the Milnor algebra of the singularity $(X,0)$, where $J_f$ denotes the Jacobian ideal of $f$. Let $\langle f \rangle $ denote the principal ideal spanned by $f$ in $M(f)$ and $\ker m_f$ denote the kernel of the morphism $m_f:M(f) \to M(f)$ given by the multiplication by $f$. Then we know that $\langle f \rangle \subset \ker m_f$, see
\cite{BrS74}. Moreover, one has $\dim_\C ( \langle f\rangle)=\mu-\tau$ and $\dim_\C  (\ker m_f) =\tau$.
Using this approach,  Yongqiang Liu has shown  in \cite{Li17} that 
$$\tau \geq \frac{1}{2}\mu.$$
He asked  there which values can take the quotient
$$\rho:=\rho(X,0)=\mu(X,0)/ \tau(X,0).$$
The obvious inequality $\tau \leq \mu$ and Theorem \ref{thmA} (3) show that 
$$1 \leq \rho(X,0)<2$$
when $(X,0)$ is  non smooth. It also shows that the inclusion of ideals $\langle f \rangle \subset \ker m_f$ is strict when $(X,0)$ is not  a smooth germ.

To construct singularities $(X,0)$ with a large quotient $\rho(X,0)$ is not easy, since the Tjurina number $\tau(X,0)$ is difficult to compute in general, e.g. since it is not a topological invariant it cannot be expressed in terms of Puiseux pairs. 

\begin{ex} \label{ex0}
There is a sequence of isolated plane curve singularities $(X_m,0)$ such that the sequence of rational numbers $\rho(X_m,0)$ is strictly increasing with limit $4/3$. Moreover, the singularities can be chosen to be all either irreducible, or consisting of smooth branches with distinct tangents.

In the irreducible case, consider the sequence of singularities
$$(X_m,0): f=x^{2m+1}+x^my^{m+1}+y^{2m}=0.$$
Then the associated projective plane curve of degree $d=2m+1$
$$ C: x^{2m+1}+x^my^{m+1}+y^{2m}z=0$$
is free with exponents $(d_1,d_2)=(m,m)$, see \cite[Theorem 1.1]{DSt17}.
This implies that 
$$\tau=\tau(X_m,0)=\tau(C)=(d-1)^2-d_1d_2=3m^2,$$
see \cite[Equation (2.2)]{DSt17}.
Since clearly $(X_m,0)$ is a semi-weighted homogeneous singularity, it follows that $\mu=\mu(X_m,0) =2m(2m-1)$, and hence the claim follows in this case.

In the case of singularities consisting of smooth branches with distinct tangents, consider the sequence
$$(X_m,0): f=x^{2m+1}+y^{2m+1}+x^{m+1}y^{m+1}.$$
Again $(X_m,0)$ is a semi-weighted homogeneous singularity, and from that we get
$\mu=\mu(X_m,0) =4m^2$.
To determine the Tjurina number, note that the monomials $x^ay^b$ for
$0 \leq a, b \leq 2m-1$ form a basis for the Milnor algebra $M(f)$.
The Euler formula implies that the monomial $x^{m+1}y^{m+1}$ belongs to the ideal $(f) \subset M(f)$. To get a basis for the Tjurina algebra
$T(f)=M(f)/(f)$ of $f$, we have to discard from the above basis all the multiples of $x^{m+1}y^{m+1}$, namely $(m-1)^2$ elements.
It follows that $\tau=\tau(X_m,0)=4m^2-(m-1)^2$, which yields the claim in this case as well.
\end{ex}

\begin{question} \label{q0}
Is it true that 
$$\rho(X,0)=\mu(X,0)/ \tau(X,0) < \frac{4}{3}$$
for any isolated plane curve singularity?
\end{question}

\end{document}